\documentclass{amsart}
\usepackage{amsfonts}
\usepackage{amssymb}
\usepackage{amsxtra}
\usepackage{amstext}
\usepackage{graphicx}

\DeclareMathOperator{\area}{area}
\DeclareMathOperator{\cov}{cov}
\DeclareMathOperator{\im}{Im}

\newcommand{\R}{\mathbb{R}}
\newcommand{\C}{\mathbb{C}}

\newcommand{\N}{\mathbb{N}}
\newcommand{\E}{\mathbb{E}}
\newcommand{\Z}{\mathbb{Z}}

\newcommand{\pp}{\mathbb{P}}

\newcommand{\kC}{\mathcal{C}}

\newcommand{\kE}{\mathcal{E}}
\newcommand{\kG}{\mathcal{G}}

\newcommand{\kS}{\mathcal{S}}

\newcommand{\kT}{\mathcal{T}}
\newcommand{\kU}{\mathcal{U}}

\newcommand{\excu}{\mathcal{E}}
\newcommand{\smex}{\mathcal{E}_{\text{sm}}}
\newcommand{\mdex}{\mathcal{E}_{\text{md}}}
\newcommand{\lgex}{\mathcal{E}_{\text{lg}}}
\newcommand{\anglechange}[3]{w^{#1}_{#3}(#2)}  
\renewcommand{\angle}[3]{\psi^{#1}_{#3}(#2)}
\renewcommand{\log}{\ln}

\newtheorem {theo} {Theorem} [section]

\newtheorem {lem}[theo] {Lemma}
\newtheorem {prop}[theo] {Proposition}
\newtheorem {thm}[theo] {Theorem}
\newtheorem {cor}[theo] {Corollary}

\title[Windings of planar random walks and averaged Dehn function]
      {Windings of planar random walks and averaged Dehn function}
\author{Bruno Schapira}
\address{D\'epartement de Math\'ematiques, B\^at. 425, Universit\'e Paris-Sud 11, F-91405 Orsay, cedex, France. }
\email{bruno.schapira@math.u-psud.fr}

\author{Robert Young}
\address{Institut des Hautes \'Etudes Scientifiques. Le Bois Marie, 35 route de Chartres, F-91440 Bures-sur-Yvette, France}
\email{rjyoung@ihes.fr}
\begin{document}

\begin{abstract}
We prove sharp estimates on the expected number of windings of a simple random walk on the square or triangular lattice. This gives new lower bounds on the averaged Dehn function, which measures the expected area needed to fill a random curve with a disc.
\end{abstract}
\keywords{simple random walk, winding number, averaged Dehn function}

\subjclass[2000]{52C45; 60D05}

\maketitle

\section{Introduction}

The winding numbers of random curves have received much study, going
back to L\'{e}vy \cite{Levy} and Spitzer \cite{Sp}.  L\'{e}vy, in
particular, studied the L\'{e}vy area of a planar Brownian motion,
which is based on adding, with sign, the winding number of the curve
around different regions of the plane.  In this paper, we will study
the total winding number of a random closed curve, that is, the
integral of the magnitude of the winding number over the plane.  The
total winding number of a curve is always non-negative, unlike the
L\'{e}vy area, and it is connected to problems of filling curves by
discs or cycles.

Let $\theta_t(z)$ be the angular part of a Brownian motion with
respect to $z$.  Spitzer showed that $(\ln t)^{-1}\theta_t(z)$
converges to the Cauchy distribution, so the winding angle of a
Brownian motion can be quite large.  Further results on the
winding number of Brownian motion suggest that the total winding
number of a loop based on Brownian motion is infinite.  Werner
\cite{W} showed that if $\theta_t(z)$ is the angular part of a
Brownian motion with respect to $z$, then
\begin{equation}
\label{limitindexBM}
k^2\area\{z\in \C\mid \theta_t(z)-\theta_0(z)\in [2\pi k, 2\pi(k+1))\} \to t/2\pi
\end{equation}
in $L^2$ as $k\to +\infty$. Thus if $\gamma$ is the loop formed by
connecting the ends of a Brownian motion by a straight line and
$i_\gamma(x)$ is the winding number of $\gamma$ around $x$, then
$$\E\left[\int_{\R^2} |i_\gamma(x)|\;dx\right]=\infty.$$
Similarly, Yor \cite{Y} explicitly
computed the law of the index of a
point $z\in \C$ with respect to a Brownian loop.  Using this result,
one can show that there is equality in \eqref{limitindexBM} for all
$k\neq 0$, so the expectation above is also infinite if $\gamma$ is a
Brownian loop.

This infinite area, however, is due to the presence of small regions
with arbitrarily large winding number, which are not present if
$\gamma$ is the curve formed by connecting the ends of a random walk
by a straight line.  In fact, since $\gamma$ has finite length, its
total winding number is finite.  Let $i_n(z)$ be the random variable
corresponding to the number of times that a random walk of $n$ steps
winds around $z$.  Windings of
random walks have been studied less than windings of Brownian curves,
but one result is an analogue of Spitzer's theorem for the random
walk, due to B\'elisle \cite{Be}, who proved that for any fixed $z\in
\C$, the distribution of $i_n(z)/\ln n$ converges to the hyperbolic
secant distribution as $n\to +\infty$.

In this paper, we will prove 
\begin{theo}
\label{thm:maintheo}
Let $(S_n, n\ge 0) $ be the simple random walk on
the unit square lattice in the complex plane.  Let $\kS_n$ be the loop joining the
points $S_0,S_1,\dots,S_n,S_0$
in order by straight lines and for $z\in \C\setminus \kS_n$, let
$i_n(z)$ be the index of $z$ with respect to $\kS_n$.  Then 
\begin{equation*}
\E\left[\int_\C |i_n(z)|\ dz\right] \sim \frac{1}{2 \pi} n\ln \ln n,
\end{equation*}
when $n\to +\infty$.
\end{theo}
Note that the integral in the theorem is well defined since $\kS_n$ has Lebesgue measure $0$. 

The basic idea of the proof is to use strong approximation to relate
windings of the random walk to windings of Brownian motion.  The main
effect of replacing a Brownian curve with a random walk is to
eliminate points with very high winding numbers, and the replacement
does not affect points which are far from the Brownian motion.  We
thus find a lower bound on the expected winding number by considering
points which are far from the random walk (Prop.~\ref{prop:low
  index}).  We find an upper bound by bounding the number of points
with high winding numbers (Prop.~\ref{prop:high index}).  Windings
around a point can be broken into classes depending on their distance
from the point, and we bound the number of windings in each class
separately.  We will give some notation and preliminaries in Sec.~\ref{sec:prelims}, and
prove Thm.~\ref{thm:maintheo} in Sec.~\ref{sec:proof}.  

In Section~\ref{sec:apps}, we describe an application to geometric
group theory.  The isoperimetric problem is a classical problem in
geometry which asks for the largest area that can be bounded by a loop of
a given length.  This question can be asked for a variety of spaces,
and is particularly important in geometric group theory, where the
growth rate of this area as a function of the length of the loop is
known as the Dehn function and carries information on the geometry of
a group (see \cite{Bri} for a survey).  Gromov \cite[5.A$'_6$]{Gr}
proposed studying the distribution of areas of random curves as an
alternative to studying the supremal area of a curve, and the total
filling area of a curve in $\R^2$ is bounded below by its total
winding number.



Finally, in Section \ref{secext}, we give some extensions of
Thm.~\ref{thm:maintheo}, including a version that holds for the
triangular lattice and a lower bound for Brownian bridges in the plane.

\vspace{0.2cm} \noindent \textit{Acknowledgments: We would like to thank Anna Erschler, Jean-Fran\c{c}ois Le Gall, David Mason, Pierre Pansu, Christophe Pittet, Oded Schramm, and Wendelin Werner for useful discussions and suggestions.}

\section{Preliminaries and notation}\label{sec:prelims}

We start by laying out some of the background and notation for the
rest of the paper.  We will recall some standard notation and describe
results of Zaitsev on approximating random walks by Brownian motion
and of Werner on winding numbers of Brownian motion.

We first describe some notation for the growth of functions.  Recall
that $f(x)=O(g(x))$ if and only if $\limsup_{x\to \infty}
|f(x)/g(x)|<\infty$.  We define a class of quickly-decaying functions:
\begin{equation*}
\kG=\{g:\R^+ \to \R^+ \mid g(k)=O(k^{-c}) \quad \forall c>0\}.
\end{equation*}

We next describe some notation for planar random walks and winding
numbers.  Throughout this paper, we will identify $\R^2$ with $\C$.
Let $(S_i,i\ge 0)$ be a random walk on $\R^2$ with i.i.d.\ bounded
increments, $\E[S_{i}-S_{i-1}]=0$, and $\cov (S_i-S_{i-1})=\kappa I$. Except express mention of the contrary, we will assume that $S$ is the simple random walk on the unit square lattice, for which $\kappa=1/2$.  For
any $n$, consider the rescaled process $(X_t,0\le t \le 1)$ defined by:
$$X_t:=\frac{S_{\lfloor nt\rfloor}+(t-\frac{\lfloor nt\rfloor}{n})(S_{\lfloor nt\rfloor+1}-S_{\lfloor nt\rfloor})}{\sqrt{\kappa n}},$$
for all $t\ge 0$ (the dependence on $n$ in the notation will often be
implicit).  Here, $\lfloor x\rfloor$ represents the largest integer
which does not exceed $x$.  Denote by $\kC_n$ the loop made of the curve
$(X_t,0\le t \le 1)$ and the segment joining $X_0$ and $X_1$.  This
curve connects the points of the random walk in order and joins the
endpoints. Note in comparison with the notation of Theorem
\ref{thm:maintheo} that $\kC_n=\kS_n/\sqrt{n/2}$.  We will primarily
work with $X_t$ and $\kC_n$ rather than $S_i$ and $\kS_n$.  

If $Y_t$ is a continuous function of $t\in [0,1]$ and $z$ is not in its image,
let $\angle{z}{t}{Y}$ be the unique continuous lift of $\im
(\ln(Y_t-z))$ such that $\angle{z}{0}{Y}\in [0,2\pi)$.  If
$T=\bigcup_i I_i$ is a finite disjoint union of intervals $I_i$ with
endpoints $x_i$ and $y_i$, let $\anglechange{z}{T}{Y}=\sum_i
[\angle{z}{y_i}{Y}-\angle{z}{x_i}{Y}]$.  If $z$ is in the image of
$Y$, set $\anglechange{z}{T}{Y}=0$.  For $z$ not in the image of $\kC_n$, let
\begin{eqnarray}
\label{jnz}
j_n(z):=\left[\frac{\anglechange{z}{[0,1]}{X}}{2\pi}\right],
\end{eqnarray}
where $[x]$ represents the closest integer to $x$. Note that if $\anglechange{z}{[0,1]}{X}$ is an odd
multiple of $\pi$, then $z$ is on the line connecting $X_0$ and $X_1$. So if $z\notin \kC_n$, then $j_n(z)$ is well-defined, and this is the index of $z$ with
respect to $\kC_n$. Actually it will be also convenient to define $j_n(z)$ when $z$ is on the line connecting $X_0$ and $X_1$. In this case we set by convention
\begin{eqnarray}
\label{jnz2}
j_n(z):= \frac{\anglechange{z}{[0,1]}{X}}{2\pi}-\frac 1 2.
\end{eqnarray}
Observe now that 
$$i_n(z\sqrt{n/2})=j_n(z) \quad \hbox{for all }z\in \C\setminus
\kC_n.$$

One of our main tools is the fact that a
random walk can be approximated by a Brownian motion. A
strong approximation theorem due to Zaitsev \cite{Zaitsev}, which
improves bounds of Einmahl \cite{Ein} and generalizes results of 
Koml\'os, Major and Tusn\'ady \cite{KMT}, implies the following:
\begin{thm}[\cite{Zaitsev}]\label{thm:zaitsev}
  If $(X_t, 0\le t \le 1)$ is defined as above, then there is a constant $c>0$ such
  that for any $n>1$, there exists a coupling of $(X_t,0\le t \le 1)$
  with a Brownian motion $(\beta_t,0\le t \le 1)$ such that
  $$\pp \biggl[\sup_{k\le n,k\in \N}|X_{k/n} - \beta_{k/n}| \ge c\frac{\ln n}{\sqrt{n}}\biggr] \le \frac{1}{n^4}.$$
\end{thm}

Standard properties of the Brownian motion imply that there is a
constant $c'>0$ such that
$$\pp \biggl[\sup_{t\le 1} \sup_{h\le 1/n} |\beta_t-\beta_{t+h}|\ge \frac{\ln n}{\sqrt{n}}\biggr] \le \frac{c'}{n^4},$$
for all $n > 1$, so there are constants $c_0>0$ and $c''>0$ such that 
\begin{equation*}
\pp \left[\sup_{t\le 1} |X_t - \beta_t| \ge c_0\frac{\ln n}{\sqrt{n}}\right] \le \frac{c''}{n^4},
\end{equation*}
for all $n\ge 1$.  Let 
$$\epsilon_n:= c_0\frac{\ln n}{\sqrt{n}}$$
and let $\kU$ be the event
\begin{equation*}
\kU=\kU(n):=\left\{\sup_{t\le 1} |X_t - \beta_t| \le \epsilon_n\right\}.
\end{equation*}

\vspace{0.2cm}
\noindent Let $(\beta_t,t\ge 0)$ be a complex Brownian motion starting from $0$.  As with the random walk, we can connect the endpoints of the
Brownian motion by a line segment to form a closed curve
$\tilde{\kC}$.  If $z\notin \tilde{\kC}$, let
$$\tilde{j}(z):=\left[\frac{\anglechange{z}{[0,1]}{\beta}}{2\pi}\right],$$
be its index with respect to $\tilde{\kC}$.

We will need to consider the number of times $\beta$ winds
around a point, especially when this occurs outside a ball of
radius $\epsilon$.  
Let $B(z,r)$ denote the closed disc with center $z\in \C$ and radius $r$.  For
$r\ge 0$, let
\begin{align}
\label{Trtr}
T_r(z):=\inf \{s\ge 0 \mid \beta_s\not \in B(z,r)\}\\
\nonumber t_r(z):=\inf \{s\ge 0 \mid \beta_s\in B(z,r)\}.
\end{align}
For any $z\neq 0$, we have the following
skew-product representation (see for instance \cite{LG} or \cite{RY}):
$$\beta_t-z=\exp(\rho^z_{A^z_t}+i\theta^z_{A^z_t}),$$
where $((\rho^z_t,\theta^z_t),t\ge 0)$ is a two-dimensional Brownian motion, and 
$$A^z_t=\int_0^t\frac{1}{|\beta_s-z|^2}\ ds \quad \hbox{for all } t\ge 0.$$
Note the intuition behind this representation; when $\beta_t$ is far from $z$, then $A^z$ increases slowly, so that $\arg (\beta_t-z)$ also varies slowly. 
For $\epsilon >0$ and $t\ge 0$, let 
\begin{equation}\label{eq:zepsilon}
  Z_\epsilon(t) := \left(\int_0^t\frac{1_{\{|\beta_s-z|\ge \epsilon\}}}{|\beta_s-z|^2}\ ds\right)^{1/2},
\end{equation}
and let $Z_\epsilon := Z_\epsilon(1).$
This $Z_\epsilon$ controls the amount of winding around $z$ which
occurs while the Brownian motion is outside $B(z,\epsilon)$.  
The next lemma is essentially taken from \cite{Sp2} and  Lemma
$2$ and Corollary $3$ (ii) of \cite{W}. It
shows, among other things, that $Z_\epsilon$ is not likely to be much larger than $|\ln \epsilon|$. 
\begin{lem}\label{lem:prelim}
\ \\
\begin{enumerate} 
\item[(i)] \label{lem:wernerasympbnd} There is a $g\in \kG$, such that for all $\epsilon \in (0,1/2)$ and all $k\ge 1$, 
$$\pp[Z_\epsilon \ge k]\le g(k/|\ln \epsilon|).$$
\item[(ii)] \label{lem:nearness} There exists a function $\phi \in L^1$, such that for all $\epsilon \in (0,1/2)$ and all $z\neq 0$,

$$\pp[t_\epsilon(z) \le 1] \le \frac{\phi(z)}{|\ln \epsilon|}.$$
\item[(iii)] There exists a function $\phi \in L^1$, such that for all $\epsilon \in (0,1/2)$, all $z\neq 0$ and all $k\ge 1$, 
$$\pp[Z_\epsilon \ge k] \le \frac{\phi(z)}{k}.$$ \label{lem:wernerunivbd}
\end{enumerate}
\end{lem} 
\begin{proof} We start with part (i). The proof is essentially contained in the proof of Lemma $2$ in \cite{W}.  Let $M=e^{\sqrt{k}}$. Consider first the case where $|z|\le M/2$. 
The skew-product decomposition shows that $Z_\epsilon(T_M(z))^2$ has the distribution of the exit time of $[0,\ln (M/\epsilon)]$ by a reflected Brownian motion 
starting from $\ln |z|-\ln \epsilon$ if $|z| > \epsilon$, or from $0$ if $|z|\le \epsilon$. 
The Markov property implies that this exit time is dominated by the exit time $\sigma_{\ln (M/\epsilon)}$ of $[0,\ln (M/\epsilon)]$ by a reflected Brownian motion starting from $0$.  
Thus 
\begin{align*}
\pp[Z_\epsilon(T_M(z)) \ge k] &\le  \pp [\sigma_{\ln (M/\epsilon)} \ge k^2]\\
&\le  \pp\left[\sigma_1 \ge \frac{k^2}{(\sqrt{k} + |\ln \epsilon|)^2}\right]\\
&\le  c\exp\left( -c'\frac{k^2}{(\sqrt{k} + |\ln \epsilon|)^2}\right),
\end{align*}
where $c$ and $c'$ are positive constants.  For the last inequality
see for instance \cite{PS} Proposition 8.4.  If $k/|\ln \epsilon|>1$,
\begin{align*}
\pp[Z_\epsilon(T_M(z)) \ge k] & \le c\exp\left( -c'\frac{k/|\ln \epsilon|}{(|\ln \epsilon|^{-1/2} + \sqrt{|\ln \epsilon|/k})^2}\right) \\
& \le c\exp\left( -c' \frac{k}{10|\ln \epsilon|}\right) \le g(k/|\ln \epsilon|),
\end{align*}
for some $g\in \kG$.

\noindent On the other hand, since $|z|\le M/2$, by the maximal inequality, there is a $g\in \kG$ such that
$$\pp[T_M(z)\le 1] \le \pp[T_{M/2}(0)\le 1]\le g(k).$$
The case $|z|\le M/2$ now follows from the inequality
$$\pp[Z_\epsilon \ge k]\le \pp[T_M(z) \le 1]+ \pp[Z_\epsilon(T_M(z)) \ge k].$$

\noindent Next assume that $|z|>M/2$. Then  
\begin{align*}
\pp[Z_\epsilon \ge k]& \le  \pp\left[\inf_{s\le 1} |\beta_s-z|\le k^{-1}\right]\\
                                                &\le  \pp[T_{M/2}(0)\le 1] \le g'(k),
\end{align*} 
for some function $g'\in \kG$, which concludes the proof of (i).  

\vspace{0.2cm}
\noindent Part (ii) is essentially due to Spitzer \cite{Sp2} (see also \cite{LG} for a precise statement). Part (iii) is a special case
of Corollary $3$ (ii) in \cite{W}.
\end{proof}
\noindent Let $\epsilon_n$ be as in the remarks after
Thm.~\ref{thm:zaitsev}.  The consequence of (ii) in the previous lemma for the random walk is the 
\begin{cor}\label{cor:RWnearness} For any $c>0$, there exists a function $\phi \in
  L^1$, such that for all $z\neq 0$ and sufficiently large $n$,
$$\pp[\inf \{t\ge 0 \mid X_t\in B(z,c \epsilon_n)\} \le  1] \le \frac{\phi(z)}{\ln n}.$$
\end{cor}
\begin{proof}
  By Lemma \ref{lem:prelim} (ii), there is a $\phi\in
  L^1$ and a $c'>0$ such that for sufficiently large $n$,
  \begin{align*}
    \pp[\inf \{t\ge 0 \mid X_t\in B(z,c \epsilon_n)\} \le  1] & \le \pp[t_{(c+1) \epsilon_n}(z) \le  1,\kU]+\pp[\kU^c]\\
    & \le \frac{ \phi(z)}{\ln n}+c' n^{-4}.
  \end{align*}
  Since the length of a step of the random walk is bounded, if $n$
  is sufficiently large and $|z|>n+1$, then $\pp[\inf \{t\ge 0 \mid X_t\in B(z,c
  \epsilon_n)\} \le 1]=0$.  We thus find
  $$\pp[\inf \{t\ge 0 \mid X_t\in B(z,c \epsilon_n)\} \le  1] \le \frac{ \phi(z)}{\ln n}+c'\min\left\{n^{-4},\frac{|z|^{-3}}{n}\right\} \le \frac{\phi'(z)}{\ln n},$$
  as desired.
\end{proof}

\section{Proof of Theorem \ref{thm:maintheo}} \label{sec:proof} 
Let $D\subset \C$ be given by
$$D=\{z_1+i z_2\mid z_1\sqrt{n/2},z_2\sqrt{n/2}\not \in \Z \text{ for any $n$}\}.$$  
This is the set of points which are not on any of the edges of the
rescaled lattices for the random walk, and $\C\setminus D$ has measure $0$.

\vspace{0.2cm}
\noindent Theorem \ref{thm:maintheo} follows by dominated convergence from the
following:
\begin{prop}\label{prop:mainprop}
Let $(S_n, n\ge 0) $ be the simple random walk on
the unit square lattice in the complex plane. For $z\in D$, let $j_n(z)$ be defined by \eqref{jnz} and \eqref{jnz2}. Then
$$\lim_{n\to +\infty} \frac{\E\bigl[\bigl|j_n(z)\bigr|\bigr]}{\ln \ln n} = \frac{\int_0^1 p_s(0,z)\;ds}{\pi}\quad \hbox{for all }z\in D,$$
where $p_s(0,z)=(2\pi s)^{-1} \exp(-|z|^2/2s)$. Moreover there is a function $\phi\in L^1$ such that for all sufficiently large $n$,
$$\frac{\E\bigl[\bigl|j_n(z)\bigr|\bigr]}{\ln \ln n} \le \phi(z)\quad \hbox{for all }z\in D. $$
\end{prop}

\subsection{Sketch of proof}

The main idea of our proof is that there are no small windings in
the rescaled random walk because the granularity of the walk makes it impossible to
approach a point more closely than roughly $(\kappa n)^{-1/2}$.  We
thus expect the number of windings of $\kC_n$ to be roughly the
number of windings of the Brownian motion which stay far from $z$
(relative to $(\kappa n)^{-1/2}$).  This has the effect of eliminating
points with large winding numbers, since points with many windings
generally have many close windings.  We will show that
the number of points with winding number $\le \ln n$ remains roughly
the same, but there are many fewer points with winding number $>\ln
n$.

Accordingly, we will bound the total winding number from below by
considering just the points with small winding number:
\begin{prop}[Points with low index]\label{prop:low index}
  Let $(S_i,i\ge 0)$ be a random walk on $\R^2$ with i.i.d.\ bounded
  increments, $\E[S_{i}-S_{i-1}]=0$, and $\cov (S_i-S_{i-1})=\kappa
  I$.  Let $X_t$ and $j_n(z)$ be as in Sec.~\ref{sec:prelims}.  Then
  $$ \lim_{n\to +\infty} \frac{1}{\ln \ln n} \sum_{k=1}^{\ln n} \pp[|j_n(z)|\ge k] = \frac{1}{\pi}\int_0^1 p_s(0,z)\;ds \quad \hbox{for all }z\in D,$$
  and there is a $\phi\in L^1$ such that for all sufficiently large $n$, 
  $$\frac{1}{\ln \ln n} \sum_{k=1}^{\ln n} \pp[|j_n(z)|\ge k] \le \phi(z) \quad \hbox{for all }z\in D.$$
\end{prop}
Werner's \cite{W} results (see Lemma~\ref{lem:Wernerlem} in the next
subsection) imply that the proposition also holds for the Brownian
motion; that is, when $j_n(z)$ is replaced by $\tilde{j}(z)$, the
number of times that the Brownian motion winds around $z$.  This gives
the lower bound required by Prop.~\ref{prop:mainprop}.

In Subsection~\ref{sec:low index}, we prove this proposition through
strong approximation.  Since the random walk is usually close to the
Brownian motion, $j_n(z)$ and $\tilde{j}(z)$ can differ only if $z$ is
close to the Brownian motion.  Using Spitzer's estimate of the area of
the Wiener sausage, we will show that most of the points with winding
numbers $\le \ln n$ lie far from the curve.  

This proposition relies mainly on strong approximation and not on
properties of the random walk, so we can prove it for random walks
with arbitrary increments.  Furthermore, it can be proved using a
weaker embedding theorem, such as the Skorokhod embedding theorem.

We get an upper bound by showing that there are few points with
winding number $>\ln n$.  In fact, we show that
\begin{prop}\label{prop:high index}
  Under the hypotheses of Proposition \ref{prop:mainprop}, there is a function
  $\phi \in L^1$ such that for all $k\ge \ln n$, 
  $$\pp[|j_n(z)|\ge k]\le \frac{\phi(z)}{k} \quad \hbox{for all }z\in D.$$

  \noindent Furthermore, there is a $\phi'\in L^1 $ such that for any $\epsilon>0$,
  any $k\ge (\log n)^{1+\epsilon}$ and $n$ large enough, 
  $$\pp[|j_n(z)|\ge k]\le \frac{\ln n}{k^2}\phi(z)\quad \hbox{for all }z\in D.$$
\end{prop}
The proof of the proposition proceeds by decomposing the random walk
into pieces that are close to and far from a given point $z$; this
technique is similar to that used by B\'elisle \cite{Be}.  We bound
the amount of winding accumulated by the faraway pieces using strong
approximation and Lemma \ref{lem:prelim}.  We bound the winding of the
nearby pieces by noting that the random walk usually spends little
time near $z$, and so does not accumulate a large winding number while
close to $z$.  This is different from the case of the Brownian motion,
in which most of the winding around a point with high winding number
occurs very close to the point.

The proof of this proposition requires stronger machinery than the
proof of the lower bound.  In particular, it requires that
$S_n$ be the simple random walk on the square grid.  Furthermore, we
need the full power of Thm.~\ref{thm:zaitsev}, since
Thm.~\ref{thm:zaitsev} allows us to choose $\epsilon_n$ so that the
$\epsilon_n$-neighborhood of $z$ contains roughly $(\ln n)^2$ grid
points.

Assuming the two propositions, we can prove Prop.~\ref{prop:mainprop}.
\begin{proof}[Proof of Proposition \ref{prop:mainprop}]
Note that by summation by parts,
\begin{align}
\label{ingred1}
\E\bigl[\bigl|j_n(z)\bigr|\bigr] & =\sum_{k=-\infty}^{+\infty} |k| \pp[j_n(z)=k] \\
\nonumber &=\sum_{k=1}^{+\infty} \pp[|j_n(z)|\ge k].
\end{align}
Moreover Proposition \ref{prop:high index} implies that for $n$ large enough,
\begin{eqnarray}
\label{ingred3}
\frac{1}{\ln \ln n}\sum_{k=\ln n}^{\infty} \pp[|j_n(z)|\ge k]\le \epsilon \phi'(z) + \frac{\phi(z)}{\ln\ln n}\quad \hbox{for all }z\in D.
\end{eqnarray}
Proposition \ref{prop:mainprop} now immediately follows from
\eqref{ingred1}, Proposition \ref{prop:low index} and
\eqref{ingred3}. \end{proof} 
We will prove Proposition \ref{prop:low index} in the next subsection and Proposition \ref{prop:high index}
in subsections \ref{sec:decomp}--\ref{sec:largeexc}.

\subsection{Points with low index}
\label{sec:low index}
We concentrate here on Proposition \ref{prop:low index}. The proof combines
strong approximation \cite{Zaitsev}, Spitzer's estimate on the area of the
Wiener sausage \cite{Sp2}, and Werner's estimate of
$\pp[\tilde{j}(z)=k]$ \cite{W}.
Recall that Werner showed that
\begin{lem}[\cite{W}, Lemma 5]\label{lem:Wernerlem}\ \\
  \begin{enumerate}
  \item For all $t>0$ and all $z\neq 0$,
    $$x^2 \pp[\anglechange{z}{[0,t]}{\beta}\in [x,x+2\pi]]\mathop{\to}^{x\to \infty} 2\pi \int_0^t p_s(0,z)\; ds.$$
  \item For all $t>0$, there is a function $\phi_t\in L^1$ such that for all $z\neq 0$ and all $x$ large enough,
    $$x^2 \pp[\anglechange{z}{[0,t]}{\beta}\in [x,x+2\pi]]\le \phi_t(z).$$
  \end{enumerate}
\end{lem}


\vspace{0.3cm} \noindent \textit{Proof of Proposition \ref{prop:low index}}:
For $\epsilon>0$, let
$W_\epsilon$ be the $\epsilon$-neighborhood of the loop $\tilde{\kC}$ formed by connecting the endpoints of the Brownian motion.
Since $W_\epsilon$ is mostly
made up of the Wiener sausage of $\beta$, Spitzer's estimate
\cite{Sp2} (see also \cite{LG}) shows that for all $z$, 
\begin{equation}
\label{sausage}
\lim_{\epsilon\to 0} |\ln \epsilon| \pp[z\in W_\epsilon] = \pi \int_0^1 p_s(0,z)\;ds,
\end{equation}
and that there is a $\phi\in L^1$ such that
\begin{equation}
\label{sausagephi}
\pp[z\in W_\epsilon] \le \frac{\phi(z)}{|\ln \epsilon|}.
\end{equation}
Observe that on the set $\kU$, the straight-line homotopy between $\tilde{\kC}$ and $\kC_n$
lies entirely in $W_{\epsilon_n}$. So if $z\not\in W_{\epsilon_n}$,
then $j_n(z)=\tilde{j}(z)$.  Thus
$$\bigl|\pp[|j_n(z)|\ge k,\ \kU]-\pp[|\tilde{j}(z)|\ge k,\ \kU]\bigr|\le \pp[z\in W_\epsilon].$$
On the other hand, if $j_n(z)\ne \tilde{j}(z)$, then one of them is non-zero, so $|z|\le \sup_{0\le t\le 1} |X_t|$ or $|z|\le \sup_{0\le t\le 1} |\beta_t|$. Therefore
\begin{align*}
\bigl|\pp[|j_n(z)|\ge k,\ \kU^c]-\pp[|\tilde{j}(z)|\ge k,\ \kU^c]\bigr| & \le \pp\left[\sup_{0\le t\le 1} |X_t|>|z|,\ \kU^c\right]+\pp\left[\sup_{0\le t\le 1} |\beta_t|>|z|,\ \kU^c\right] \\
&\le \pp[\kU^c]1_{\{|z|<n\}}+ \min \left\{ \pp\left[\sup_{0\le t\le 1} |\beta_t|>|z|\right],\ \pp[\kU^c]\right\}.
\end{align*}
Since $\pp[\kU^c]=O(n^{-4})$, there is a $\phi'\in L^1$ such that
$$\bigl|\pp[|j_n(z)|\ge k,\ \kU^c]-\pp[|\tilde{j}(z)|\ge k,\ \kU^c]\bigr|\le \frac{\phi'(z)}{n}.$$
Thus
\begin{eqnarray}
\label{pjnz}
\bigl|\pp[|j_n(z)|\ge k]-\pp[|\tilde{j}(z)|\ge k]\bigr|\le \pp[z\in W_{\epsilon_n}]+\frac{\phi'(z)}{n},
\end{eqnarray}
and by \eqref{sausage},
$$ \biggl|\sum_{k=1}^{\ln n}\pp[|j_n(z)|\ge k]-\sum_{k=1}^{\ln n}\pp[|\tilde{j}(z)|\ge k]\biggr|\le (\ln n) \pp[z\in W_{\epsilon_n}]+\frac{\phi'(z)\ln n}{n}$$
is bounded. Thus 
\begin{align*}
\lim_{n\to +\infty} \frac{1}{\ln \ln n} \sum_{k=1}^{\ln n} \pp[|j_n(z)|\ge k]&=\lim_{n\to\infty} \frac{1}{\ln \ln n}\sum_{k=1}^{\ln n}\pp[|\tilde{j}(z)|\ge k]\\
&=\frac{1}{\pi}\int_0^1 p_s(0,z)\;ds,
\end{align*}
where the last equality follows from Lemma \ref{lem:Wernerlem}. 
Finally, by \eqref{pjnz}, \eqref{sausagephi} and Lemma \ref{lem:Wernerlem}, there is a $\phi''\in L^1$ such that
$$\pp[|j_n(z)|\ge k]\le \frac{\phi''(z)}{k}+\frac{\phi''(z)}{\ln n},$$
and so
$$\sum_{k=1}^{\ln n}\pp[|j_n(z)|\ge k]\le \phi''(z)\log\log n+\phi''(z).$$
This concludes the proof of Proposition \ref{prop:low index}. \hfill $\square$

\subsection{Decomposing the winding number} \label{sec:decomp} In
the next subsections, we will prove Proposition \ref{prop:high index}.

To bound the probability that a point will have large index, we will
introduce a decomposition of the winding number into small, medium,
and large windings, depending on their distance from the point.  We
will then use different methods to bound the number of windings.
During large windings, the random walk stays away from the point and
is well-approximated by the Brownian motion.  The number of medium and small
windings is bounded by the fact that the random walk spends little
time near the point, and while the bounds on the medium and large
windings only use strong approximation, the bound on the small windings
uses the full power of Thm.~\ref{thm:zaitsev}.

Define the stopping times $(\tau_i, i\ge 0)$ and $(\sigma_i,i\ge 1)$ by 
\begin{align*}
\tau_0&:=0, \\
\sigma_i&:=\inf \{t\ge \tau_{i-1}\mid |\beta_t-z|<2\epsilon_n\}\quad \forall i\ge 1,\\
\tau_i&:=\inf\{t\ge \sigma_i \mid |\beta_t-z|>4\epsilon_n\}\quad \forall i\ge 1.
\end{align*}
These times divide the curve into pieces close to $z$ and far from $z$.  If $t\in
[\sigma_i,\tau_i)$, then $|\beta_t-z|\le 4\epsilon_n$; if $t\in
[\tau_i,\sigma_{i+1})$, then $|\beta_t-z|\ge 2\epsilon_n$. Let 
$$\kT:= \bigcup_{i\ge 0}[\tau_i,\sigma_{i+1})\cap [0,1];$$
this is a.s.\ a finite union of intervals during which $|\beta_t-z|\ge 2\epsilon_n$.

We also break the random walk into excursions with winding number $\pm
1/2$; this decomposition relies on the assumption that $(S_i,i\ge 0)$ is a
nearest-neighbor walk on a square lattice, and this is the main step
that requires this assumption.  Let
$\hat{z}$ be the center of the square in the rescaled lattice
containing $z$; since $z\in D$, this is unique.  Let
$$\Delta_z^\pm:=\hat{z}\pm(1+i)\R^+$$
be the halves of the diagonal line
through $\hat{z}$.  If the random walk hits $\Delta_z^+$ before
$\Delta_z^-$, define
\begin{align*}
e_0&:=\inf\{t\ge 0 \mid X_t \in \Delta_z^+\}\\
e_{2i+1}&:=\inf\{t\ge e_{2i} \mid X_t \in \Delta_z^-\}\\
e_{2i}&:=\inf\{t\ge e_{2i-1} \mid X_t \in \Delta_z^+\};
\end{align*}
otherwise, define the $e_i$'s with $+$ and $-$ switched.  Note that
$e_i\in\Z/n$ for all $i$.  We call intervals of the form
$[e_i,e_{i+1}]$ excursions with respect to $z$, and let $\kE=\kE(z)$ be the set of all such excursions
which are subsets of $[0,1]$.  If $e=[t_1,t_2]$ is an excursion, set 
$w(e):=\anglechange{z}{[t_1,t_2]}{X}/2\pi=\pm 1/2$.  Then
$$\left|\frac{\anglechange{z}{[0,1]}{X}}{2\pi}-\sum_{e\in \excu}w(e)\right|\le 2.$$

We will classify these excursions as small, medium, or large.  Let
$$\smex=\smex(z):=\{[u,v]\in \excu \mid |X_u-z|\le \epsilon_n\},$$
be the set of
excursions starting close to $z$; we call them {\em small excursions}. Let
$$\lgex=\lgex(z):=\{[u,v]\in \excu \mid  [u,v] \subset \kT\}.$$
These {\em large
  excursions} stay far from $z$.  On the set $\kU$, these two sets
are disjoint.  Let 
$$\mdex=\mdex(z):=\excu\setminus(\smex\cup\lgex)$$
be the set of {\em medium
  excursions}, so that every excursion is either small, medium, or
large.

Proposition \ref{prop:high index} is a consequence of the following lemmas.
\begin{lem}
\label{lem:smallmed}
There exists a function $\phi \in L^1$, such that for all sufficiently large $n$, all $k\ge 1$, and all $z\in D$,
\begin{enumerate}
\item[(i)] $\displaystyle \pp\biggl[\biggl|\sum_{e\in \smex}w(e)\biggr|\ge k\biggr]\le \frac{\ln n}{k^2}\phi(z).$ \label{lem:smallmed:small}
\item[(ii)] $\displaystyle \pp\biggl[\biggl|\sum_{e\in \mdex}w(e)\biggr|\ge k\biggr]\le \frac{\ln n}{k^2}\phi(z).$
\end{enumerate}
\end{lem}

\begin{lem}
\label{lem:large}
\ \\
\begin{enumerate}
\item[(i)] \label{lem:large:>C} There is a function $\phi \in L^1$, such
  that for any $\epsilon>0$, all sufficiently large $n$ and $k\ge (\log n)^{1+\epsilon}$,
  $$\pp\biggl[\biggl|\sum_{e\in \lgex}w(e)\biggr|\ge k\biggr]\le \frac{\ln n}{k^2}\phi(z)\quad \hbox{for all }z\in D.$$
\item[(ii)] There exists a function $\phi \in L^1$, such that for all
  sufficiently large $n$ and all $k\ge 1$,
  $$\pp\biggl[\biggl|\sum_{e\in \lgex}w(e)\biggr|\ge k\biggr]\le \frac{\phi(z)}{k}\quad \hbox{for all }z\in D.$$ 
\end{enumerate}
\end{lem}
We will prove these lemmas in the next three subsections.
\subsection{Small excursions}

In this subsection, we will prove Lemma \ref{lem:smallmed} (i) by using
results on the occupation times of points by a random walk on a
lattice.

We first show that the symmetry of the random walk implies that
excursions are equally likely to go clockwise or counterclockwise.
For the simple random walk, this follows from the fact that reflecting
an excursion across the line $\Delta^+_z\cup \Delta^-_z$ results in an
equally probable excursion which winds in the opposite direction, but
there is also an argument which relies only on the random walk being
symmetric.  Let $[e_i,e_{i+1}]$ be an excursion from $\Delta^+_z$ to
$\Delta^-_z$ and let
$$t_0=\max\{t\in [e_i,e_{i+1}] \mid X_t\in \Delta^+_z\}.$$
The portion of the excursion between $t_0$ and $e_{i+1}$ travels
between $p_1=X_{t_0}$ and $p_2=X_{e_{i+1}}$.  Its time reversal
travels from $p_2$ to $p_1$.  Rotating this curve by 180 degrees and
translating gives a curve going from $p_1$ to $p_2$ which winds around
$z$ in the opposite direction of the original.  Replacing the section
of the excursion between $t_0$ and $e_{i+1}$ by this rotated curve is
a measure-preserving map on the space of random walk paths.  This map
doesn't change $\excu$ or $\smex$ and changes the direction of the
excursion from clockwise to counterclockwise or vice versa, so the
variables $w([e_i,e_{i+1}])$ are independent and take values from $\{-1/2,1/2\}$
with equal probability.  Furthermore,
$$\E\biggl[\Bigl(\sum_{e\in\smex}w(e)\Bigr)^2\biggr]=\frac{1}{4}\E[\#\smex].$$


Many points have no small excursions; if the random walk never comes
close to $z$, then $\smex=\emptyset$.  Let $\tau=\inf \{t\ge 0 \mid
X_t\in B(z,2 \epsilon_n)\}$.  Then by Corollary
\ref{cor:RWnearness},
$$\pp[\tau \le  1] \le \frac{\phi(z)}{\ln n},$$
for some $\phi \in L^1$.  Applying the Markov property, this gives
$$\pp\biggl[\biggl|\sum_{e\in\smex}w(e)\biggr|\ge k\biggr] \le \frac{1}{4} \frac{\E[\#\smex]}{k^2}\frac{\phi(z)}{\ln n},$$
where the expectation on the right hand side is taken conditionally on
the past before time $\tau$.  Since each small excursion
starts at a point in $(\Delta_z^+ \cup \Delta_z^-)\cap
B(z,\epsilon_n)$, the number $\#\smex$ of small excursions is bounded
by the number of visits to such points; it is well known that the mean
number of visits to each site is bounded by $c \ln n$ for some
constant $c$.  Since the number of lattice points in $(\Delta_z^+ \cup \Delta_z^-)\cap
B(z,\epsilon_n)$ is of order $\ln n$, we get 
$$\E[\#\smex]\le c (\ln n)^2,$$
for some constant $c>0$, which
establishes Lemma \ref{lem:smallmed} (i).

\subsection{Medium excursions}

On the set $\kU$, medium excursions start outside of
$B(z,\epsilon_n)$ and enter $B(z,5\epsilon_n)$ at some point.  We
will divide the medium excursions into two classes, depending on
whether they start inside or outside the ball $B(z,8\epsilon_n)$, and
bound the number in each class.

Let 
$$\mdex'=\mdex'(z):=\{[t_1,t_2] \in \mdex \mid |X_{t_1}-z|>8\epsilon_n\},$$
be
the set of medium excursions starting outside $B(z,8\epsilon_n)$ and 
$$\mdex''=\mdex''(z):=\{[t_1,t_2] \in \mdex \mid |X_{t_1}-z|\le 8\epsilon_n\},$$
be its complement.
Given $0<a<b$, define the sequences of stopping times
\begin{align*}
\tau^{a,b}_0&:=0, \\
\sigma^{a,b}_i& = \sigma^{a,b}_i(z):=\inf \{t\ge \tau^{a,b}_{i-1}\mid |\beta_t-z|<a\epsilon_n\}\quad \forall i\ge 1,\\
\tau^{a,b}_i&=\tau^{a,b}_i(z):=\inf\{t\ge \sigma^{a,b}_i \mid |\beta_t-z|>b\epsilon_n\}\quad \forall i\ge 1.
\end{align*}
These stopping times record the number of times 
the Brownian motion crosses the annulus $B(z,b\epsilon_n)\setminus
B(z,a\epsilon_n)$.  Using the skew-product decomposition, we can bound
the number of such crossings that occur before time 1:

\begin{lem} 
  \label{lem:crossings} 
  For any $0<a<b$, there are a constant $c>0$ and functions $\phi\in L^1$ and $g\in \kG$,
  such that for sufficiently large $n$ and all $z\in \C$,
  $$\pp[\sigma^{a,b}_i<1]\le \frac{\phi(z)}{\ln n}\left(1-\frac{c}{\ln n}\right)^i+g(n)\le \frac{\phi(z)}{\ln n}e^{-ci/\ln n}+g(n).$$
\end{lem}
\begin{proof}
  Recall the definitions of $T_r(z)$ and $t_r(z)$ from \eqref{Trtr}. 
  By the skew-product decomposition, there is a constant
  $c>0$ (depending on $a$ and $b$) such that a Brownian motion starting on the circle of radius
  $b\epsilon_n$ around $z$ has probability $p\le 1-c/\ln n$ of
  entering $B(z,a\epsilon_n)$ before escaping $B(z,n)$.  Thus, if
  $\sigma^{a,b}_i<T_n(z)$, then a.s.  $\tau^{a,b}_i< T_n(z)$ and 
  $$\pp[\sigma^{a,b}_{i+1}<T_n(z)\mid \sigma^{a,b}_{i}<T_n(z)]= p.$$
  Thus
  $$\pp[\sigma^{a,b}_{i}<T_n(z)\mid \sigma^{a,b}_1<T_n(z)]= p^i.$$
  Furthermore it is well known that
  $$\pp[T_n(z)<1]\le g(n),$$
for some $g\in \kG$. 
 Now $\sigma^{a,b}_1=t_{a\epsilon_n}(z)$ and by Lemma \ref{lem:prelim} (ii), there is a $\phi\in L^1$ such that for all sufficiently large $n$,
  $$\pp[\sigma^{a,b}_1<1]= \pp[t_{a\epsilon_n}(z)<1]\le \frac{\phi(z)}{\ln n}.$$
  We conclude that
  $$\pp[\sigma^{a,b}_{i}<1]\le \pp[\sigma^{a,b}_{i}<T_n(z),\sigma^{a,b}_1<1]+g(n)\le \frac{\phi(z)}{\ln n}p^i+g(n),$$
  as desired.  
\end{proof}

On $\kU$, the Brownian motion crosses the annulus
$B(z,7\epsilon_n)\setminus B(z,6\epsilon_n)$ during each excursion in
$\mdex'$.  Thus, there are $\phi,\phi'\in L^1$ such that for all sufficiently large $n$,
\begin{align}
\label{kkkk}
\pp[\#\mdex'>k]& \le \pp[\sigma^{6,7}_{k}<1]+\pp[\kU^c] \\
 \nonumber              & \le \frac{\phi(z)}{\ln n} e^{-ck/\ln n}+O(n^{-4}) \\
 \nonumber              & \le \phi'(z)\frac{\ln n}{k^2}+O(n^{-4}).
\end{align}
for all $k$.  

To bound the cardinality of $\mdex''$, we will show that if $\mdex''$ is
large, then the rescaled random walk likely crosses the annulus
$$B(z,11\epsilon_n)\setminus B(z,8\epsilon_n)$$ many times before time
$1$.  First note that, on $\kU$, 
$$\mdex''\subset \bar{\kE} := \{[t_1,t_2] \in \excu \mid \epsilon_n\le |X_{t_1}-z|\le 8\epsilon_n\}.$$
Using strong approximation and the skew-product decomposition, one can
show that the probability that a random walk started in
$$(B(z,8\epsilon_n)\setminus B(z,\epsilon_n))\cap \Delta^+_z$$
leaves $B(z,11\epsilon_n)$ before hitting $\Delta^-_z$ is bounded away
from zero, say by $p>0$.  Let $\eta$ be the number of times the random
walk crosses from inside $B(z,8\epsilon_n)$ to outside
$B(z,11\epsilon_n)$.  Let $\{x_i\}$ be a sequence of independent random variables with $\pp[x_i=1]=p$,
$\pp[x_i=0]=1-p$.  Since the starting times of excursions in
$\bar{\kE}$ form a sequence of stopping times,
$$\pp[\#\bar{\kE}\ge 2k/p,\eta \le k]\le \pp\biggl[\sum_{i=1}^{2k/p} x_i \le k\biggr] \le e^{-c k},$$
for some $c>0$. So by reproducing the steps of \eqref{kkkk}, we get
\begin{align*}
  \pp[\#\bar{\kE}\ge 2k/p] 
   \le \phi(z)\frac{\ln n}{k^2}+ O(n^{-4}),
\end{align*}
for some $\phi\in L^1$.

Combining the bounds on $\#\mdex'$ and
$\#\mdex''$, we see that there is some $\phi\in L^1$ such that
$$\pp\biggl[\biggl|\sum_{e\in \mdex} w(e)\biggr|\ge \left(1+\frac{2}{p}\right)k\biggr] \le \pp\left[\#\mdex\ge \left(1+\frac{2}{p}\right)k\right] \le \phi(z)\frac{\ln n}{k^2}+O(n^{-4}).$$
We can eliminate the $O(n^{-4})$ term using an argument like the one in the proof of
Corollary \ref{cor:RWnearness}.  That is, the number of medium
excursions is at most $n$, and if $|z|> \sqrt{n/\kappa}$, then the random walk
does not come near $z$. So
\begin{align*}
\pp\biggl[\biggl|\sum_{e\in \mdex} w(e)\biggr|\ge
\left(1+\frac{2}{p}\right)k\biggr]&\le \phi(z)\frac{\ln n}{k^2}+c n^{-4}1_{\{|z|\le
\sqrt{n/\kappa}\}} 1_{\{(1+2/p)k\le n\}} \\
&\le \phi'(z)\frac{\ln n}{k^2},
\end{align*}
for some $c>0$ and $\phi'\in L^1$,
as desired.  This proves Lemma \ref{lem:smallmed} (ii).

\subsection{Large excursions}\label{sec:largeexc}
In this subsection, we will prove Lemma \ref{lem:large} using strong approximation.

Note first that on the set $\kU$, the winding of the Brownian
motion during $\kT$ is approximated by the winding number of the large
excursions.  The difference between the two arises from excursions which intersect the beginning or end of an interval of $\kT$; each such interval can thus increase the difference by at most $1$.  Thus on the set $\kU$,
$$\biggl|\sum_{e\in\lgex}w(e)\biggr| \le \left|\frac{\anglechange{z}{\kT}{\beta}}{2\pi}\right|+\#\{i\mid \tau_i \le 1\}+1,$$
and
\begin{equation}\label{eq:lgexcursionpbd}
\pp\biggl[\biggl|\sum_{e\in\lgex}w(e)\biggr|>2k +1\biggr] \le \pp\left[\left|\frac{\anglechange{z}{\kT}{\beta}}{2\pi}\right|\ge k\right]+\pp\left[\#\{i\mid \tau_i \le 1\}>k\right]+\pp[\kU^c].
\end{equation}
By Lemma \ref{lem:crossings}, we have
\begin{equation} 
\label{eqtaui}
\pp[\#\{i\mid \tau_i \le 1\}>k]\le \frac{\ln n}{k^2}\phi(z)+g(n),
\end{equation}
for some $\phi \in L^1$ and $g\in \kG$.  To prove Lemma \ref{lem:large}, it thus suffices to bound 
$$\pp\left[\left|\frac{\anglechange{z}{\kT}{\beta}}{2\pi}\right|\ge k\right].$$

Let 
$$\widetilde{Z} := \left(\int_{\kT}\frac{ds}{|\beta_s-z|^2}\right)^{1/2}.$$
Since $\kT$ is a finite union of intervals whose endpoints are
stopping times for the Brownian motion, $\anglechange{z}{\kT\cap[0,t]}{\beta}$
is a continuous local martingale whose quadratic variation at time $1$ is $\widetilde{Z}^2$. Thus $\anglechange{z}{\kT}{\beta}$ is equal in law to $\gamma \widetilde{Z}$, with $\gamma$ a standard
normal variable.  If $Z_{\epsilon_n}$ is given by \eqref{eq:zepsilon},
then $\widetilde{Z} \le Z_{\epsilon_n}$, and we get
$$\pp\left[|\anglechange{z}{\kT}{\beta}|\ge k\right] \le  \pp\left[|\gamma| Z_{\epsilon_n} \ge k\right].$$

We first show part (i) of the lemma.  Let $\delta>0$ be such that $(1-\delta)(1+\epsilon)>1.$
Then by Lemma \ref{lem:prelim} (i), for $n$ sufficiently large depending on $\epsilon$ and $k\ge (\ln n)^{1+\epsilon}$, there
are $c>0$ , $c'>0$, $\epsilon'>0$ and $g'$, $g''\in \mathcal{G}$ such that
\begin{align*}
\pp\left[|\gamma| Z_{\epsilon_n} \ge k\right] &\le \pp\left[|\gamma| \ge k^{\delta} \right]+ \pp\left[Z_{\epsilon_n} \ge k^{1-\delta}\right]\\
&\le c' e^{-k^{2\delta}/2} + g(k^{1-\delta}/|\ln \epsilon_n|)\\
&\le c' e^{-k^{2\delta}/2} + g'(c k^{\epsilon'})\le g''(k).
\end{align*}
Thus for sufficiently large $n$ and $k\ge (\ln n)^{1+\epsilon}$,
\begin{align*}
\pp[|\anglechange{z}{\kT}{\beta}|\ge k]\le k^{-4}.
\end{align*}
Moreover, if
$|\anglechange{z}{\kT}{\beta}|\ge \pi$, then we must have
$T_{|z|/2}(0)\le 1$; this has probability less than $e^{-d|z|^2}$ for
some constant $d>0$. By interpolation this gives
\begin{equation} 
\label{eqwzT}
\pp\left[|\anglechange{z}{\kT}{\beta}| \ge k\right]\le k^{-2}e^{-d|z|^2/2}.
\end{equation}
Thus by \eqref{eq:lgexcursionpbd}, \eqref{eqtaui}, and \eqref{eqwzT} there is a $\phi\in L^1$, independent of $\epsilon$, such that
for all sufficiently large $n$ and $k\ge (\ln n)^{1+\epsilon}$,
$$\pp\biggl[\biggl|\sum_{e\in\lgex}w(e)\biggr|\ge k\biggr] \le \frac{\ln n}{k^2}\phi(z)+O(n^{-4}).$$
As in the previous subsection, we can
eliminate the $O(n^{-4})$ term by changing $\phi.$ This proves part
(i) of Lemma \ref{lem:large}.

For part (ii), we use Lemma \ref{lem:prelim} (iii).
If $n$ is sufficiently large, we have
\begin{align*}
\pp\left[|\gamma| Z_{\epsilon_n} \ge k\right] &=\int_0^\infty \frac{2}{\sqrt{2\pi}} e^{-x^2/2} \pp\left[Z_{\epsilon_n} \ge \frac{k}{x}\right]\;dx\\
& \le \int_0^\infty \frac{2}{\sqrt{2\pi}} e^{-x^2/2} \frac{\phi(z) x}{k}\;dx=\frac{2}{\sqrt{2\pi}}\frac{\phi(z)}{k},
\end{align*}
for some $\phi\in L^1$.  Thus
$$\pp\biggl[\biggl|\sum_{e\in\lgex}w(e)\biggr|>k\biggr] \le \frac{\phi'(z)}{k}+O(n^{-4}),$$
for some $\phi'\in L^1$, and as above, the $O(n^{-4})$ term can be absorbed into $\phi'$.

\section{Application to the study of Dehn functions} 
\label{sec:apps}
One motivation for studying the winding numbers of random walks comes
from geometric group theory, namely the study of Dehn functions.  For
an introduction to Dehn functions with rigorous definitions, see
\cite{Bri}; we will only sketch the definitions.  Given a closed curve
$\gamma$ in a space $X$, one can ask for the infimal area of a disc
with boundary $\gamma$.  We call this the filling area of
$\gamma$, denoted $\delta(\gamma)$, and define a function $\delta$ so
that $\delta(n)$ is the supremal filling area of curves in $X$ with
length at most $n$; this is the Dehn function of $X$.

Gromov noted that when a group $G$ acts properly discontinuously and
cocompactly on a connected, simply connected space $X$, the filling
area of curves can be described in terms of $G$ \cite{Gr}.  Roughly,
formal products of generators of $G$ (words) correspond to curves in
$X$ and conversely, curves in $X$ can be approximated by words.  For
example, if $X$ is the universal cover of a compact space $M$ and
$G=\pi_1(M)$, generators of $G$ correspond to closed curves in $M$.  A
formal product $w$ of generators corresponds to a concatenation of
these curves, and if $w$ represents the identity in $G$, this
concatenation lifts to a closed curve in $X$.  A disc whose boundary
is this curve corresponds to a way of using the relators of $G$ to
show that $w$ represents the identity.  We define the Dehn function of
a group similarly to that of a space; if $w$ is a word, let $\bar{w}$
be the element of $G$ represented by $w$, and if $\bar{w}$ is the
identity, let the {\em filling area} $\delta(w)$ of $w$ be the minimal
number of applications of relators necessary to reduce $w$ to the
identity.  Define the {\em Dehn function} $\delta$ of $G$ so that
$\delta(n)$ is the maximal filling area of words with length at most
$n$.  The Dehn function of a group depends {\em a priori} on the
presentation of the group, but one can show that the growth rate of
$\delta$ is a group invariant and that the Dehn function of $G$ grows
at the same rate as that of $X$.

The Dehn function measures maximal filling area; Gromov \cite[5.A$'_6$]{Gr} also
proposed studying the distribution of $\delta(w)$ when $w$ is a random
word of length $n$.  This has led to multiple versions of an averaged
Dehn function or mean Dehn function; see \cite{BogVen} for some
alternatives to our definition.  We will give  definitions based on
random walks.  Given a random walk on $G$ supported on the generators,
paths of the random walk correspond to words; in particular, the
random walk induces a measure on the set of words of length $n$.  We
can thus define the {\em averaged Dehn function} $\delta_{\text{avg}}(n)$ to
be the expectation of $\delta(w)$ with respect to this measure,
conditioned on the event that $w$ represents the identity in $G$.  If
the random walk has a nonzero probability of not moving at each step,
this is defined for all $n$.  The averaged Dehn function is not known
to be a group invariant and it is possible that its growth depends on
the transition probabilities of the random walk.

The averaged Dehn function often behaves differently than the Dehn
function.  For example, in $\Z^2$, the Dehn function is quadratic,
corresponding to the fact that a curve of length $n$ in $\R^2$
encloses an area at most quadratic in $n$.  On the other hand, random
walks in $\Z^2$ enclose much less area, and the averaged Dehn function
is at most $O(n\ln n)$ for a wide variety of random walks.  A similar
phenomenon occurs in nilpotent groups; if $\delta(n)=O(n^k)$ for some
$k>2$, then $\delta_{\text{avg}}(n)=O(n^{k/2})$ for many random walks
\cite{Yo}.

Since the averaged Dehn function involves bridges of random walks on
groups, it is often difficult to work with.  M.\ Sapir has proposed an
alternative, the {\em random Dehn function} \cite{BogVen}, which depends on a
choice of a random walk and a choice of a word $v_x$ for every
element $x\in G$ so that $v_x$ represents $x^{-1}$.  If $w$ is a
word, $w v_{\bar{w}}$ is a word representing the identity; essentially,
the $v_x$ give a way to close up any path.  One can then define
$$\delta_{\text{rnd}}(n)=\E[\delta(w v_{\bar{w}})]$$
where $w$ is chosen from the set of words of length $n$ with measure
corresponding to the random walk.  This function depends on the choice
of $v_x$, but if the $v_x$ are short enough, the choice may not affect
the function.  In particular, Proposition \ref{prop:low index} implies the following:
\begin{prop}
\label{propdelta}
  For every $x\in \Z^d$, $d\ge 2$, let $v_x$ be a word of minimal length
  representing $x^{-1}$.  Then there is a $c>0$ such that the random
  Dehn function of $\Z^d$ with respect to the $v_x$ and any random
  walk on $\Z^d$ with i.i.d.\ bounded increments,
  $\E[S_{i}-S_{i-1}]=0$, and whose support generates $\Z^d$ satisfies
  $$\delta_{\text{rnd}}(n)\ge c n \ln \ln n$$
  for sufficiently large $n$.
\end{prop}
\begin{proof}
  To use Proposition \ref{prop:low index}, we need to use the correspondence
  between words in $\Z^d$ and curves in $\R^d$.  If $w=w_1\dots w_n$
  is a word in some generating set for $\Z^d$, let $\widetilde{w}$ be
  the curve in $\R^d$ connecting
  $$\overline{w_1}, \overline{w_1w_2},\dots,\overline{w_1\dots w_n}\in \Z^d\subset \R^d$$
  by straight lines.

  It suffices to show that 
  $$\E[\delta(\widetilde{w v_{\bar{w}}})]\ge c n\ln\ln n.$$ 
  Winding numbers provide a lower bound on the filling area of curves
  in the plane; if $\gamma:S^1\to \R^2$, and $f:D^2\to \R^2$ is a
  PL or smooth map with boundary $\gamma$, then $\#f^{-1}(x)\ge |i_n(x)|$
  away from a set of measure zero and thus
  $$\area{f}\ge \int_{\R^2} |i_\gamma(x)|\;dx.$$
  Indeed, $\int_{\R^2} |i_\gamma(x)|\;dx$ is the abelianized area of
  $w v_{\bar{w}}$, a function defined in \cite{BaMSh} as a lower bound
    for the filling area.

  As in the 2-dimensional case, we can define a random closed curve
  $\kS$ in $\R^d$ which consists of a random walk and a line
  connecting the endpoints.  Let $p:\R^d\to \R^2$ be the projection onto
  the first two coordinates; since this map is area-decreasing,
  $$\delta(\kS)\ge \int_{\R^2} |i_{p(\kS)}(x)|\;dx.$$

  The curve $\kS$ can be obtained from $\widetilde{w v_{\bar{w}}}$ by
  replacing $\widetilde{v_{\bar{w}}}$ with a straight line.  For every
  $x\in \Z^d$, let $f_x:D^2\to \R^d$ be a disc whose boundary consists
  of $v_x$ and a straight line connecting $0$ and $x$ and such that
  $\area(f_x)\le 4\ell(v_x)^2$.  If $h:D^2\to \R^d$ is a disc with
  boundary $\widetilde{w v_{\bar{w}}}$, we can construct a disc with
  boundary $\kS$ by adjoining $f_{\bar{w}}$ to $h$, and so
  $$\delta(\widetilde{w v_{\bar{w}}})\ge \int_{\R^2} |i_{p(\kS)}(x)|\;dx-4\ell(v_w)^2.$$
  Then Proposition \ref{prop:low index} implies that
  $$\E[\delta(\widetilde{w v_{\bar{w}}})]\ge c n\ln\ln n-4\E[\ell(v_w)^2]\ge \frac{c}{2} n\ln\ln n$$
  for sufficiently large $n$, as desired.
\end{proof}

\section{Extensions to bridges and other random walks}
\label{secext}
We can prove a similar result to Proposition \ref{propdelta} for the averaged Dehn function of $\Z^2$
by extending Proposition \ref{prop:low index} to the random walk
bridge.  To prove Proposition \ref{prop:low index}, it suffices to
have strong approximation of the random walk bridge by a Brownian
bridge, estimates of $\E[\tilde{j}(z)]$, and estimates of the
probability of lying near the loop; all of these results exist in the
literature.  Borisov \cite{Bor} has proven a strong approximation
theorem which applies to the simple random walk bridge in $\Z^2$, and
the analogue of Werner's result for the Brownian loop is a consequence
of Yor \cite{Y}. Note that for Brownian loop, the result is more precise since \cite{Y} gives an exact formula for the expected area of points with index $k\neq 0$. The case $k=0$ has even been computed recently by Garban and Trujillo Ferreras \cite{GaTF}. 
Finally the upper bound on the
probability of being close to the Brownian loop is a consequence of
the fact that the law of the Brownian loop for $t\in[0,1/2]$ or
$t\in[1/2,1]$ is absolutely continuous with respect to the law of the
Brownian motion.  More precisely, for $\epsilon>0$, denote by
$W_\epsilon'$ the Wiener sausage associated to the Brownian loop.
Denote by $\pp$ or $\E$ the law of the Brownian motion and by $\pp'$ the law of
the Brownian loop. Now if $p_t(x,y)$ is the heat kernel, then one has
$$\pp'[z\in W'_\epsilon]\le 2 \E\left[\frac{p_{1/2}(0,\beta_{1/2})}{p_1(0,0)} 1_{\{z\in W_\epsilon\}}\right],$$
for all $\epsilon>0$. Since $p_{1/2}(0,x)$ is bounded, \eqref{sausage}
implies that
$$\pp'[z\in W'_\epsilon]\le \phi(z)/|\ln \epsilon|$$
for some $\phi\in L^1$ as desired.  Thus there is a $c>0$ such that
$\delta_\text{avg}(n)\ge c n\ln\ln n$ for sufficiently large $n$.  We
suspect that neither of these lower bounds is sharp and conjecture
that in fact $\delta_\text{avg}(n)$ and $\delta_\text{rnd}(n)$ grow
like $n\ln n$.  The integral $\int_{\R^2} |i_\gamma(x)|\;dx$ is the
area necessary to fill $\gamma$ with a chain or an arbitrary manifold
with boundary, and requiring that the filling be a disc should
increase the necessary area.

Extending Theorem \ref{thm:maintheo} to more general situations seems
difficult, because of the use of $\excu$, but it can be extended to
walks with certain symmetries.  For instance, $\excu$ can be defined
for the simple random walk on the triangular lattice.  If $z$ is a
point in the plane and $X_t$ is the rescaled triangular random walk,
let $s$ be an edge of the triangle containing $z$, let $\Delta^+$ and
$\Delta^-$ be opposite sides of the straight line containing $s$ and
define $e_i$ and $\excu$ as before.  If $e\in \excu$, define $w(e)$ to
be $0$ if $X_t$ traverses $s$ during $e$.  If this does not happen,
then $X_t$ goes around $s$ in either the positive or negative
direction; let $w(e)=\pm 1/2$ depending on the direction.  Note that
each direction is equally likely, just as in the case of the square
lattice.

Then
$$\biggl|j_n(z)-\sum_{e\in \excu} w(e)\biggr|\le 2+\frac{g_n(s)}{2}$$
where $g_n(s)$ is the number of times that $X_t$ traverses $s$.  We
can bound $\sum_{e\in \excu} w(e)$ exactly as before, by breaking
$\excu$ into small, medium, and large excursions, so it remains only
to bound $g_n(s)$.  This is straightforward; since the random walk is
$n$ steps long, $\sum_s g_n(s)=n$, and there is a $\phi\in L^1$ such
that $\E[g_n(s)]\le \phi(z)/n$.  A similar proof holds for the simple
random walk on the honeycomb lattice.

\bibliographystyle{amsalpha}
\bibliography{windingZ2}

\end{document}